\RequirePackage{fix-cm}
\documentclass[smallextended]{svjour3}       % onecolumn (second format)
\smartqed  % flush right qed marks, e.g. at end of proof
\usepackage{graphicx}
\usepackage{amsfonts}
\usepackage{amssymb}
\usepackage{amsmath}
\usepackage{algorithm}
\usepackage{algpseudocode}
\usepackage{color}

\newcounter{ass_counter}
\newtheorem{assumption}[ass_counter]{Assumption}

\newcommand{\lambdaminnz}{\lambda_{\rm min,nz}}

\newcommand{\cS}{{\cal S}}
\newcommand{\bfone}{\mathbf{1}}
\newcommand{\cond}{\mbox{\rm cond}}

\newcommand{\Lres}{L_{\mbox{\rm\scriptsize res}}}
\newcommand{\Lmax}{L_{\mbox{\rm\scriptsize max}}}
\newcommand{\beq}{\begin{equation}}
\newcommand{\eeq}{\end{equation}}

\newcommand{\R}{\mathbb{R}}

\def\sjwcommentsolved#1{}
\def\jlcommentsolved#1{}

\begin{document}

\title{Coordinate Descent Algorithms\thanks{The author was supported
    by NSF Awards DMS-1216318 and IIS-1447449, ONR Award
    N00014-13-1-0129, AFOSR Award FA9550-13-1-0138, and Subcontract
    3F-30222 from Argonne National Laboratory.}  }
% \subtitle{Do you have a subtitle?\\ If so, write it here}

\titlerunning{Coordinate Descent Algorithms}        % if too long for running head

\author{Stephen J. Wright}


\institute{Stephen J. Wright \at
              Department of Computer Sciences, University of Wisconsin-Madison, 1210 W. Dayton St., Madison, WI 53706-1685, USA \\
              Tel.: +1 608 316 4358\\
              \email{swright@cs.wisc.edu}}

\date{Received: date / Accepted: date}

\maketitle

\begin{abstract}
Coordinate descent algorithms solve optimization problems by
successively performing approximate minimization along coordinate
directions or coordinate hyperplanes. They have been used in
applications for many years, and their popularity continues to grow
because of their usefulness in data analysis, machine learning, and
other areas of current interest. This paper describes the fundamentals
of the coordinate descent approach, together with variants and
extensions and their convergence properties, mostly with reference to
convex objectives.  We pay particular attention to a certain problem
structure that arises frequently in machine learning applications,
showing that efficient implementations of accelerated coordinate
descent algorithms are possible for problems of this type. We also
present some parallel variants and discuss their convergence
properties under several models of parallel execution.

\end{abstract}

\section{Introduction}
\label{sec:intro}

Coordinate descent (CD) algorithms for optimization have a history
that dates to the foundation of the discipline. They are iterative
methods in which each iterate is obtained by fixing most components of
the variable vector $x$ at their values from the current iteration,
and approximately minimizing the objective with respect to the
remaining components. Each such subproblem is a lower-dimensional
(even scalar) minimization problem, and thus can typically be solved
more easily than the full problem.

CD methods are the archetype of an almost universal approach to
algorithmic optimization: solving an optimization problem by solving a
sequence of simpler optimization problems. The obviousness of the CD
approach and its acceptable performance in many situations probably
account for its long-standing appeal among
practitioners. Paradoxically, the apparent lack of sophistication may
also account for its unpopularity as a subject for investigation by
optimization researchers, who have usually been quick to suggest
alternative approaches in any given situation. There are some very
notable exceptions. The 1970 text of Ortega and
Rheinboldt~\cite[Section~14.6]{OrtR70} included a comprehensive
discussion of ``univariate relaxation,'' and such optimization
specialists as Luo and Tseng~\cite{LuoT92a,LuoT93a},
Tseng~\cite{Tse01a}, and Bertsekas and Tsitsiklis~\cite{BerT89} made
important contributions to understanding the convergence properties of
these methods in the 1980s and 1990s.

%% --- leading optimization specialists who investigated CD methods
%% over many years and who found it surprisingly difficult to
%% establish convergence and complexity results for the best known
%% variants, such as those that update the coordinates in cyclic
%% order.

The situation has changed in recent years. Various applications
(including several in computational statistics and machine learning)
have yielded problems for which CD approaches are competitive in
performance with more reputable alternatives. The properties of these
problems (for example, the low cost of calculating one component of
the gradient, and the need for solutions of only modest accuracy) lend
themselves well to efficient implementations of CD, and CD methods can
be adapted well to handle such special features of these applications
as nonsmooth regularization terms and a small number of equality
constraints. At the same time, there have been improvements in the
algorithms themselves and in our understanding of them. Besides their
extension to handle the features just mentioned, new variants that
make use of randomization and acceleration have been
introduced. Parallel implementations that lend themselves well to
modern computer architectures have been implemented and
analyzed. Perhaps most surprisingly, these developments are relevant
even to the most fundamental problem in numerical computation: solving
the linear equations $Aw=b$.

In the remainder of this section, we state the problem types for which
CD methods have been developed, and sketch the most fundamental
versions of CD. Section~\ref{sec:applications} surveys applications
both historical and modern. Section~\ref{sec:algs} sketches the types
of algorithms that have been implemented and analyzed, and presents
several representative convergence results. Section~\ref{sec:parallel}
focuses on parallel CD methods, describing the behavior of these
methods under synchronous and asynchronous models of computation.

Our approach throughout is to describe the CD methods in their
simplest forms, to illustrate the fundamentals of the applications,
implementations, and analysis.  We focus almost exclusively on methods
that adjust just one coordinate on each iteration. Most applications
use {\em block} coordinate descent methods, which adjust groups of
blocks of indices at each iteration, thus searching along a coordinate
hyperplane rather than a single coordinate direction.  Most derivation
and analysis of single-coordinate descent methods can be extended
without great difficulty to the block-CD setting; the concepts do not
change fundamentally.  
%% (I believe this is true of the variants described in this paper.)
We mention too that much effort has been devoted to developing more
general forms of CD algorithms and analysis, involving weighted norms
and other features, that allow more flexible implementation and allow
the proof of stronger and more general (though usually not
qualitatively different) convergence results.

\subsection{Formulations}

The problem considered in most of this paper is the following
unconstrained minimization problem:
\begin{equation} \label{eq:f}
\min_x \, f(x),
\end{equation}
where $f:\R^n \to \R$ is continuous. Different variants of CD make
further assumptions about $f$. Sometimes it is assumed to be smooth
and convex, sometimes smooth and possibly nonconvex, and sometimes
smooth but with a restricted domain. (We will make such assumptions
clear in each discussion of algorithmic variants and convergence
results.)

Motivated by recent popular applications, it is common to consider the
following structured formulation:
\begin{equation}
\label{eq:reg}
\min_{x} \,  h(x) := f(x) + \lambda \Omega(x),
\end{equation}
where $f$ is smooth, $\Omega$ is a regularization function that may be
nonsmooth and extended-valued, and $\lambda>0$ is a regularization
parameter. $\Omega$ is often convex and usually assumed to be
separable or block-separable.  When separable, $\Omega$ has the form
\begin{equation} \label{eq:reg.sep}
\Omega(x) = \sum_{i=1}^n \Omega_i(x_i).
\end{equation}
where $\Omega_i: \R \to \R$ for all $i$. The best known examples of
separability are the $\ell_1$-norm (in which $\Omega(x) = \|x\|_1$ and
hence $\Omega_i(x_i) = |x_i|$) and box constraints (in which
$\Omega_i(x_i) = I_{[l_i,u_i]}(x_i)$ is the indicator function for the
interval $[l_i,u_i]$).  Block separability means that the $n \times n$
identity matrix can be partitioned into column submatrices $U_i$,
$i=1,2,\dotsc,N$ such that
\begin{equation} \label{eq:reg.bsep}
\Omega(x) = \sum_{i=1}^N \Omega_i(U_i^Tx).
\end{equation}
Block-separable examples include group-sparse regularizers in which
$\Omega_i(z_i) := \|z_i\|_2$.  Formulations of the type
\eqref{eq:reg}, with separable or block-separable regularizers, arise
in such applications as compressed sensing, statistical variable
selection, and model selection.

%% An extension of interest is to add a single linear constraint to
%% \eqref{eq:reg}, that is, $a^Tx=\beta$ for some vector $a \in \R^n$ and
%% scalar $\beta \in \R$. The dual formulation of kernel support vector
%% machines has this form, for instance, where $\Omega$ is the indicator
%% function for a set of box constraints in this instance; see for
%% example \cite{Platt99}.

The class of problems known as empirical risk minimization (ERM) gives rise
to a formulation that is particularly amenable to  coordinate descent;
see \cite{ShaZ13e}.
% \cite{LinLX14a}. 
These problems have the form
\begin{equation} \label{eq:erm}
\min_{w \in \R^d} \, \frac{1}{n} \sum_{i=1}^n \phi_i(c_i^Tw) + \lambda g(w),
\end{equation}
for vectors $c_i \in \R^d$, $i=1,2,\dotsc,n$ and convex functions
$\phi_i$, $i=1,2,\dotsc,n$ and $g$. We can express linear
least-squares, logistic regression, support vector machines, and other
problems in this framework.  Recalling the following definition of the
conjugate $t^*$ of a convex function $t$:
\begin{equation} \label{eq:conj}
t^*(y) = \sup_z (z^Ty-t(z)),
\end{equation}
we can write the Fenchel dual~\cite[Section~31]{Roc70} of
\eqref{eq:erm} as follows:
\begin{equation} \label{eq:erm.dual}
\min_{x \in \R^n} \frac1n \sum_{i=1}^n \phi_i^*(-x_i) + \lambda g^*
\left( \frac{1}{\lambda n} Cx \right),
\end{equation}
where $C$ is the $d \times n$ matrix whose columns are $c_i$,
$i=1,2,\dotsc,n$.  The dual formulation \eqref{eq:erm.dual} is has
special appeal as a target for coordinate descent, because of
separability of the summation term. 
%% As we see below, it can be handled in similarly to the separable
%% regularization term $\Omega(x)$ in \eqref{eq:reg}.

One interesting case is the system of linear equations
\begin{equation} \label{eq:Aw=b}
Aw=b, \quad \mbox{where $A \in \R^{m \times n}$},
\end{equation}
which we assume to be a feasible system. The least-norm solution is
found by solving
\begin{equation} \label{eq:kac.primal}
\min_{w \in \R^n}  \, \frac12 \|w\|_2^2 \;\; \mbox{subject to} \; Aw=b,
\end{equation}
whose Lagrangian dual is 
\begin{equation} \label{eq:kac}
\min_{x \in \R^m} \, f(x) := \frac12 \|A^Tx\|_2^2 - b^Tx.
\end{equation}
(We recover the primal solution from \eqref{eq:kac} by setting $w=A^Tx$.)
We can see that \eqref{eq:kac} is a special case of the Fenchel dual
\eqref{eq:erm.dual} obtained from \eqref{eq:erm} if we set
\[
C \leftarrow A^T, \quad
g(w) = \frac12 \|w\|_2^2, \quad
\phi_i(t_i) = I_{\{b_i\}}(t_i), \quad
\lambda=1/n,
\]
where $I_{\{b_i\}}$ denotes the indicator function for $b_i$, which is
zero at $b_i$ and infinite elsewhere.  (Its conjugate is
$I^*_{\{b_i\}}(s_i) = b_i s_i$.) The primal problem \eqref{eq:kac.primal}
can be restated correspondingly as
\[
\min_{w \in \R^n} \, \frac{1}{m} \sum_{i=1}^m  I_{\{b_i\}}(A_{i} w)  + \frac{1}{n} \| w\|_2^2,
\]
where $A_{i}$ denotes the $i$th row of the matrix $A$ in \eqref{eq:Aw=b},
which has the form \eqref{eq:erm}. 
%% By going to the dual
%% \eqref{eq:erm.dual}, and performing a change of variables, we obtain
%% \eqref{eq:kac}.

\subsection{Outline of Coordinate Descent Algorithms}

The basic coordinate descent framework for continuously differentiable
minimization is shown in Algorithm~\ref{alg:cd}. Each step consists of
evaluation of a single component $i_k$ of the gradient $\nabla f$ at
the current point, followed by adjustment of the $i_k$ component of
$x$, in the opposite direction to this gradient component. (Here and
throughout, we use $[\nabla f(x)]_i$ to denote the $i$th component of
the gradient $\nabla f(x)$.) There is much scope for variation within
this framework. The components can be selected in a cyclic fashion, in
which $i_0=1$ and 
\begin{equation} \label{eq:cyc}
i_{k+1} = [i_k \; \mbox{mod} \; n] + 1, \quad k=0,1,2,\dotsc.
\end{equation}
They can be required to satisfy an ``essentially cyclic'' condition,
in which for some $T\ge n$, each component is modified at least once in
every stretch of $T$ iterations, that is,
\begin{equation} \label{eq:ess.cyc}
\cup_{j=0}^T \{ i_{k-j} \} = \{1,2,\dotsc,n \},\quad \mbox{for all $k \ge T$.}
\end{equation}
Alternatively, they can be selected randomly  at each
iteration (though not necessarily with equal probability).  Turning to
steplength $\alpha_k$: we may perform exact minimization along the
$i_k$ component, or choose a value of $\alpha_k$ that satisfies
traditional line-search conditions (such as sufficient decrease), or
make a predefined ``short-step'' choice of $\alpha_k$ based on prior
knowledge of the properties of $f$.

\begin{algorithm} 
\caption{Coordinate Descent for \eqref{eq:f}\label{alg:cd}}
\begin{algorithmic}
\State Set $k \leftarrow 0$ and choose $x^0 \in \R^n$;
\Repeat
\State Choose index $i_k \in \{1,2,\dotsc,n\}$;
% and set $d^k = [\nabla f(x^k)]_{i_k} e_{i_k}$; 
\State $x^{k+1} \leftarrow x^k - \alpha_k [\nabla f(x^k)]_{i_k} e_{i_k}$ for some $\alpha_k>0$;
\State $k \leftarrow k+1$;
\Until termination test satisfied;
\end{algorithmic}
\end{algorithm}

The CD framework for the separable regularized problem \eqref{eq:reg},
\eqref{eq:reg.sep} is shown in Algorithm~\ref{alg:cdreg}.  At
iteration $k$, a scalar subproblem is formed by making a linear
approximation to $f$ along the $i_k$ coordinate direction at the
current iterate $x^k$, adding a quadratic damping term weighted by
$1/\alpha_k$ (where $\alpha_k$ plays the role of a steplength), and
treating the relevant regularization term $\Omega_i$ explicitly. Note
that when the regularizer $\Omega_i$ is not present, the step is
identical to the one taken in Algorithm~\ref{alg:cd}. For some
interesting choices of $\Omega_i$ (for example $\Omega_i (\cdot) = |
\cdot |$), it is possible to write down a closed-form solution of the
subproblem; no explicit search is needed. The operation of solving
such subproblems is often referred to as a ``shrink operation,''
which we denote by $S_{\beta}$ and define as follows:
\begin{equation} \label{eq:shrink}
S_{\beta}(\tau) := \min_{\chi} \frac{1}{2\beta} \| \chi - \tau \|_2^2 + \Omega_i(\chi).
\end{equation}
By stating the subproblem in Algorithm~\ref{alg:cdreg} equivalently as 
\[
\min_{\chi} \, \frac{1}{2 \lambda \alpha_k} \left\| \chi - (x^k_{i_k} - \alpha_k [\nabla f(x^k)]_{i_k}) \right\|^2 + \Omega_i(\chi),
\]
we can express the CD update as $z^k_{i_k} \leftarrow S_{\lambda
  \alpha_k}(x^k_{i_k} - \alpha_k [\nabla f(x^k)]_{i_k})$.

\begin{algorithm} 
\caption{Coordinate Descent for \eqref{eq:reg},\eqref{eq:reg.sep}\label{alg:cdreg}}
\begin{algorithmic}
\State Set $k \leftarrow 0$ and choose $x^0 \in \R^n$;
\Repeat
\State Choose index $i_k \in \{1,2,\dotsc,n\}$;
\State $z^k_{i_k} \leftarrow \arg\min_{\chi} \, 
(\chi-x^k_{i_k})^T [\nabla f(x^k)]_{i_k} + \frac{1}{2\alpha_k} \|\chi-x^k_{i_k}\|_2^2 + \lambda \Omega_i(\chi)$ for some $\alpha_k>0$;
\State $x^{k+1} \leftarrow x^k + (z^k_{i_k}-x^k_{i_k}) e_{i_k}$;
\State $k \leftarrow k+1$;
\Until termination test satisfied;
\end{algorithmic}
\end{algorithm}

Algorithms~\ref{alg:cd} and \ref{alg:cdreg} can be extended to
block-CD algorithms in a straightforward way, by updating a block of
coordinates (denoted by the column submatrix $U_{i_k}$ of the identity
matrix) rather than a single coordinate. 
% Geometrically, a block of
% coordinates corresponds to a ``coordinate hyperplane.'' 
In Algorithm~\ref{alg:cdreg}, it is assumed that the choice of block
is consistent with the block-separable structure of the regularization
function $\Omega$, that is, $U_{i_k}$ is a concatenation of several of
the submatrices $U_i$ in \eqref{eq:reg.bsep}.

\subsection{Application to Linear Equations}

For the formulation \eqref{eq:kac} that arises from the linear system
$Aw=b$, let us assume that the rows of $A$ are normalized, that is,
\begin{equation} \label{eq:Anormalized}
\| A_{i} \|_2=1 \quad \mbox{for $i=1,2,\dotsc,m$.}
\end{equation}
Applying Algorithm~\ref{alg:cd} to \eqref{eq:kac} with $\alpha_k
\equiv 1$, each step has the form
\begin{equation} \label{eq:kac.y}
x^{k+1} \leftarrow x^k - (A_{i_k} A^T x^k - b_{i_k}) e_{i_k}.
\end{equation}
If we  maintain and update the estimate $w^k$ of the solution to
the primal problem \eqref{eq:kac.primal} after each update of $x^k$,
according to $w^k = A^T x^k$, we obtain
\begin{equation} \label{eq:kac.w}
w^{k+1} \leftarrow w^k - (A_{i_k} A^T x^k - b_{i_k}) A_{i_k}^T =
w^k  -  (A_{i_k} w^k - b_{i_k}) A_{i_k}^T,
\end{equation}
which is the update formula for the Kaczmarz algorithm
\cite{kaczmarz37}.  Following this update, we have using
\eqref{eq:Anormalized} that
\[
A_{i_k} w^{k+1} = A_{i_k} w^k - (A_{i_k} w^k - b_{i_k}) = b_{i_k},
\]
so that the $i_k$ equation in the system $Aw=b$ is now satisfied. This
method if sometimes known as the ``method of successive projections''
because it projects onto the feasible hyperplane for a single
constraint at every iteration.

\subsection{Relationship to Other Methods}

Stochastic gradient (SG) methods, also undergoing a revival of
interest because of their usefulness in data analysis and machine
learning applications, minimize a smooth function $f$ by taking a
(negative) step along an estimate $g^k$ of the gradient $\nabla
f(x^k)$ at iteration $k$. It is often assumed that $g^k$ is an
unbiased estimate of $\nabla f(x^k)$, that is, $\nabla f(x^k) =
E(g^k)$, where the expectation is taken over whatever random variables
were used in obtaining $g^k$ from the current iterate
$x^k$. Randomized CD algorithms can be viewed as a special case of SG
methods, in which $g^k = n [\nabla f(x^k)]_{i_k} e_{i_k}$, where $i_k$
is chosen uniformly at random from $\{1,2,\dotsc,n\}$.  Here, $i_k$ is
the random variable, and we have
\[
E(g^k) = \frac{1}{n} \sum_{i=1}^n n [\nabla f(x^k)]_{i} e_{i} =  \nabla f(x^k),
\]
certifying unbiasedness. However, CD algorithms have the advantage
over general SG methods that descent in $f$ can be guaranteed at every
iteration.  Moreover, the variance of the gradient estimate $g^k$
shrinks to zero as the iterates converge to a solution $x^*$, since
every component of $\nabla f(x^*)$ is zero. By contrast, in general SG
methods, the gradient estimates $g^k$ may be nonzero even when $x^k$
is a solution.

The relationship between CD and SG methods can also be discerned from
the Fenchel dual pair \eqref{eq:erm} and \eqref{eq:erm.dual}. SG
methods are quite popular for solving formulation \eqref{eq:erm},
where the estimate $g^k$ is obtained by taking a single term $i_k$
from the summation and using $\nabla \phi_{i_k} (c_{i_k}^T w) c_{i_k}$
as the estimate of the gradient of the {\em full} summation. This
approach corresponds to applying CD to the dual \eqref{eq:erm.dual},
where the component $i_k$ of $x$ is selected for updating at iteration
$k$. This relationship is typified by the Kaczmarz algorithm for
$Aw=b$, which can be derived either as CD applied to the dual
formulation \eqref{eq:kac} or as SG applied to the sum-of-squares
problem 
\begin{equation} \label{eq:sos}
\min_w \, \frac12 \| Aw-b \|_2^2 =  \frac12 \sum_{i=1}^m (A_{i} w-b_i)^2.
\end{equation}

%% Another algorithm that is enjoying renewed popularity is the
%% conditional gradient algorithm, often known as ``Frank-Wolfe'' due to
%% its origins in \cite{FraW56}. When applied to a problem of minimizing
%% $f$ over a ball defined by the $\ell_1$ norm, that is,
%% \[
%% \min \, f(x) \quad \mbox{subject to $\|x \|_1 \le B$,}
%% \]
%% for some $B>0$, this method solves the following subproblem at
%% iteration $k$:
%% \[
%% \tilde{x}^k \leftarrow \arg\min_{\tilde{x}} \, \nabla f(x^k)^T(\tilde{x}-x^k) 
%% \quad \mbox{subject to $\|\tilde{x} \|_1 \le R$,}
%% \]
%% then takes a step of the form $x^{k+1} = x^k + \alpha_k (\tilde{x}^k -
%% x^k)$, for some $\alpha_k \in (0,1]$. In fact, the subproblem is
%%   solvable in closed form; we have $\tilde{x}^k = -R[\nabla
%%     f(x^k)]_{i_k} e_{i_k}$, where $i_k$ is an index corresponding to the
%%   maximum-magnitude component of $\nabla f(x^k)$. The method is this
%%   similar to a ``greedy'' version of CD, though the correspondence is
%%   not exact, because the search direction is $-R[\nabla f(x^k)]_{i_k}
%%   - x^k_{i_k}$ rather than $-[\nabla f(x^k)]_{i_k}$.

CD is related in an obvious way to the Gauss-Seidel method for $n
\times n$ systems of linear equations, which adjusts the $i_k$
variable to ensure satisfaction of the $i_k$ equation, at iteration
$k$. (Successive over-relaxation (SOR) modifies this approach by
scaling each Gauss-Seidel step by a factor $(1+\omega)$ for some
constan $\omega \in [0,1)$, chosen so as to improve the convergence
  rate.)  Standard Gauss-Seidel and SOR use the cyclic choice of
  coordinates \eqref{eq:cyc}, whereas a random choice of $i_k$ would
  correspond to ``randomized'' versions of these methods. To make the
  connections more explicit: The Gauss-Seidel method applied to the
  normal equations for \eqref{eq:Aw=b} --- that is, $A^TAw=A^Tb$ ---
  is equivalent to applying Algorithm~\ref{alg:cd} to the
  least-squares problem \eqref{eq:sos}, when the steplength $\alpha_k$
  is chosen to minimize the objective exactly along the given
  coordinate direction. SOR also corresponds to
  Algorithm~\ref{alg:cd}, with $\alpha_k$ chosen to be a factor
  $(1+\omega)$ times the exact minimum. These equivalences allow the
  results of Section~\ref{sec:algs} to be used to derive convergence
  rates for Gauss-Seidel applied to the normal equations, including
  linear convergence when $A^TA$ is nonsingular. Note that these
  results do not require feasibility of the original equations
  \eqref{eq:Aw=b}.

\section{Applications}
\label{sec:applications}

We mention here several applications of CD methods to practical
problems, some dating back decades and others relatively new. Our list
is necessarily incomplete, but it attests to the popularity of CD in a
wide variety of application communities.

Bouman and Sauer \cite{BouS96} discuss an application to positron
emission tomography (PET) in which the objective has the form
\eqref{eq:reg} where $f$ is smooth and convex and $\Omega$ is a sum of
terms of the form $|x_j-x_l|^q$ for some pairs of components $(j,l)$
of $x$ and some $q \in [1,2]$. Ye et al.~\cite{Ye:99} apply a similar
method to a different objective arising from optical diffusion
tomography.

Liu, Paratucco, and Zhang~\cite{Liu:2009:BCD:1553374.1553458} describe
a block CD approach for linear least squares plus a regularization
function consisting of a sum of $\ell_{\infty}$ norms of subvectors of
$x$. The technique is applied to semantic basis discovery, which
learns from data how to identify and classify the functional MRI
response of a person's brain when they hear certain English words.

Canutescu and Dunbrack~\cite{CanD03} describe a cyclic coordinate
descent method for determining protein structure, adjusting the
dihedral angles in a protein chain so that the atom at the end of the
chain comes close to a specified position in space.

Florian and Chen~\cite{FloC95} recover origin-destination matrices
from observed traffic flows by alternately solving a bilevel
optimization problem over two blocks of variables: the
origin-destination demands and the proportion of each
origin-destination flow assigned to each arc in the network.

Breheny and Huang~\cite{BreH11} discuss coordinate descent for linear
and logistic regression with nonconvex separable regularization terms,
reporting results for genetic association and gene expression studies.
The SparseNet algorithm~\cite{MazFH11} applied to problems with these
same nonconvex separable regularizers uses warm-started cyclic
coordinate descent as an inner loop to solve a sequence of problems in
which the regularization parameter $\lambda$ in \eqref{eq:reg} and the
parameters defining concavity of the regularization functions are
varied.

Friedman, Hastie, and Tibshirani~\cite{FriHT08a} propose a block CD
algorithm for estimating a sparse inverse covariance matrix, given a
sample covariance matrix $S$ and taking the variable in their
formulation to be a modification $W$ of $S$, such that $W^{-1}$ is
sparse. The resulting ``graphical lasso'' algorithm cycles through the
rows/columns of $W$ (in the style of block CD), solving a standard
lasso problem to calculate each update.  The same authors
\cite{FriHT10a} apply CD to generalized linear models such as linear
least squares and logistic regression, with convex regularization
terms. Their framework include such formulations as lasso, graphical
lasso, elastic net, and the Dantzig selector, and is implemented in
the package {\tt glmnet}.

Chang, Hsieh, and Lin~\cite{ChaHL08a} use cyclic and stochastic CD to
solve a squared-loss formulation of the support vector machine (SVM)
problem in machine learning, that is,
\begin{equation} \label{eq:svm.primal}
\min_w \,  \sum_{i=1}^m \max(1-y_i x_i^Tw,0)^2 + \frac{\lambda}{2} w^Tw.
\end{equation}
where $(x_i,y_i) \in \R^N \times \{0,1\}$ are feature vector / label
pairs and $\lambda$ is a regularization parameter. This problem is an
important instance of the ERM form \eqref{eq:erm}. In the best known
early application of coordinate descent to SVM, Platt~\cite{Platt99}
deals with a hinge-loss formulation of SVM, which is identical to
\eqref{eq:svm.primal} except that the square on each term of the
summation is omitted. The dual of this problem has bounds on its
variables along with a single linear constraint. Platt's procedure SMO
(for ``sequential minimal optimization''), applied to the dual,
changes two variables at a time, with the variable pair chosen
according to a ``greedy'' criterion and the search direction chosen to
maintain feasibility of the linear constraint.

Sardy, Bruce, and Tseng~\cite{SarBT00} consider the basis-pursuit
formulation of wavelet denoising:
\[
\min_x \, \frac12 \| \Phi x - y \|_2^2 + \lambda \|x\|_1.
\]
This formulation is equivalent to the well known lasso of
Tibshirani~\cite{Tib96} and has become famous because of its
applicability to sparse recovery and compressed sensing. Although this
formulation fits the ERM framework \eqref{eq:erm} and could thus be
dualized before applying CD, the approach of \cite{SarBT00} applies
block CD directly to the primal formulation.

Applications of block CD approaches to transceiver design for cellular
networks and to tensor factorization are discussed in
Razaviyayn~\cite[Section~8]{RazHL13}.

Finally, we mention several popular problem classes and algorithms
that can be interpreted as CD algorithms, but for which such an
interpretation may not be particularly helpful in understanding the
performance of the algorithm. First, we consider low-rank matrix
completion problems in which we are presented with limited information
about a rectangular matrix $M \in \R^{m \times n}$ and seek matrices
$U \in \R^{n \times r}$ and $V \in \R^{m \times r}$ (with $r$ small)
such that $UV^T$ is consistent with the observations of $M$. When the
observations satisfy a restricted isometry property (an assumption
commonly made in compressed sensing; see
\cite[Definition~3.1]{RecFP08} for a definition that applies to matrix
completion), the block CD approach of Jain, Netrapalli, and
Sanghavi~\cite[Algorithm~1]{JaiNS12a} converges to a solution. This
approach defines the objective to be the least-squares fit between the
observations and their predicted values according to the product
$UV^T$ --- a function that is nonconvex with respect to $(U,V)$ ---
and minimizes alternately over $U$ and $V$, respectively. Standard
analysis of CD for nonconvex functions would yield at best
stationarity of accumulation points, but much stronger results are
attained in \cite{JaiNS12a} because of special assumptions that are
made on the problem in this paper.

Second, we consider the ``alternating-direction method of
multipliers'' (ADMM) \cite{EckB92,BoyPCPE11a}, which has gained great
currency in the past few years because of its usefulness in solving
regularized problems in statistics and machine learning, and in
designing parallel algorithms. Each major iteration of ADMM consists
of an (approximate) minimization of the augmented Lagrangian function
for a constrained optimization problem over each block of primal
variables in turn, followed by an update to the Lagrange multiplier
estimates. It might seem appealing to do multiple cycles of updating
the primal variable blocks, in the manner of cyclic block CD, thus
finding a better approximation to the solution of each subproblem over
{\em all} primal variables and moving the method closer to the
standard augmented Lagrangian approach. Eckstein and Yao~\cite{EckY14}
show, however, that this ``approximate augmented Lagrangian'' approach
has a fundamentally different theoretical interpretation from ADMM,
and a computational comparison between the two approaches
\cite[Section~5]{EckY14} appears to show an advantage for ADMM.

\section{Coordinate Descent: Algorithms, Convergence, Implementations}
\label{sec:algs}

We now describe the most important variants of coordinate descent and
present their convergence properties, including the proofs of some
fundamental results. We also discuss the implementation of accelerated
CD methods for problems of the form \eqref{eq:erm.dual} and for the
Kaczmarz algorithm for $Aw=b$. As mentioned in the introduction, we
deal with the most elementary framework possible, to expose the
essential properties of the methods.
%%  Most papers in the literature use
%% a more general framework
%% % (involving, for example, weighted norms in
%% % place of Euclidean norms, and blocks in place of single coordinates)
%% leading to enhanced results, but our treatment below exposes most
%% of the key properties.

\subsection{Powell's Example}

We start with a simple but intriguing example due to
Powell~\cite[formula (2)]{Pow73b} of a function in $\R^3$ for which
cyclic CD fails to converge to a stationary point. The nonconvex,
continuously differentiable function $f:\R^3 \to \R$ is defined as
follows:
\begin{equation} \label{eq:powell}
f(x_1,x_2,x_3) = -(x_1 x_2 + x_2 x_3 + x_1 x_3) + \sum_{i=1}^3 (|x_i|-1)_+^2.
\end{equation}
It has minimizers at the corners $(1,1,1)^T$ and $(-1,-1,-1)^T$ of the
unit cube, but coordinate descent with exact minimization, started
near (but just outside of) one of the other vertices of the cube
cycles around the neighborhoods of six points that are close to the
six non-optimal vertices. Powell shows that the cyclic nonconvergence
behavior is rather special and is destroyed by small perturbations on
this particular example, and we can note that a randomized coordinate
descent method applied to this example would be expected to converge
to the vicinity of a solution within a few steps.  Still, this example
and others in \cite{Pow73b} make it clear that we cannot expect a
general convergence result for nonconvex functions, of the type that
are available for full-gradient descent. Results are available for the
nonconvex case under certain additional assumptions that still admit
interesting applications.  Bertsekas~\cite[Proposition~2.7.1]{Ber99}
describes convergence of a cyclic approach applied to nonconvex
problems, under the assumption that the minimizer along any coordinate
direction from any point $x$ is unique. More recent work
\cite{Att10,BolST14} focuses on CD with two blocks of variables,
applied to functions that satisfy the so-called Kurdyka-{\L}ojasiewicz
(KL) property, such as semi-algebraic functions. Convergence of
subsequences or the full sequence $\{ x^k \}$ to stationary points can
be proved in this setting.

\begin{figure}
\centering\includegraphics[width=3in]{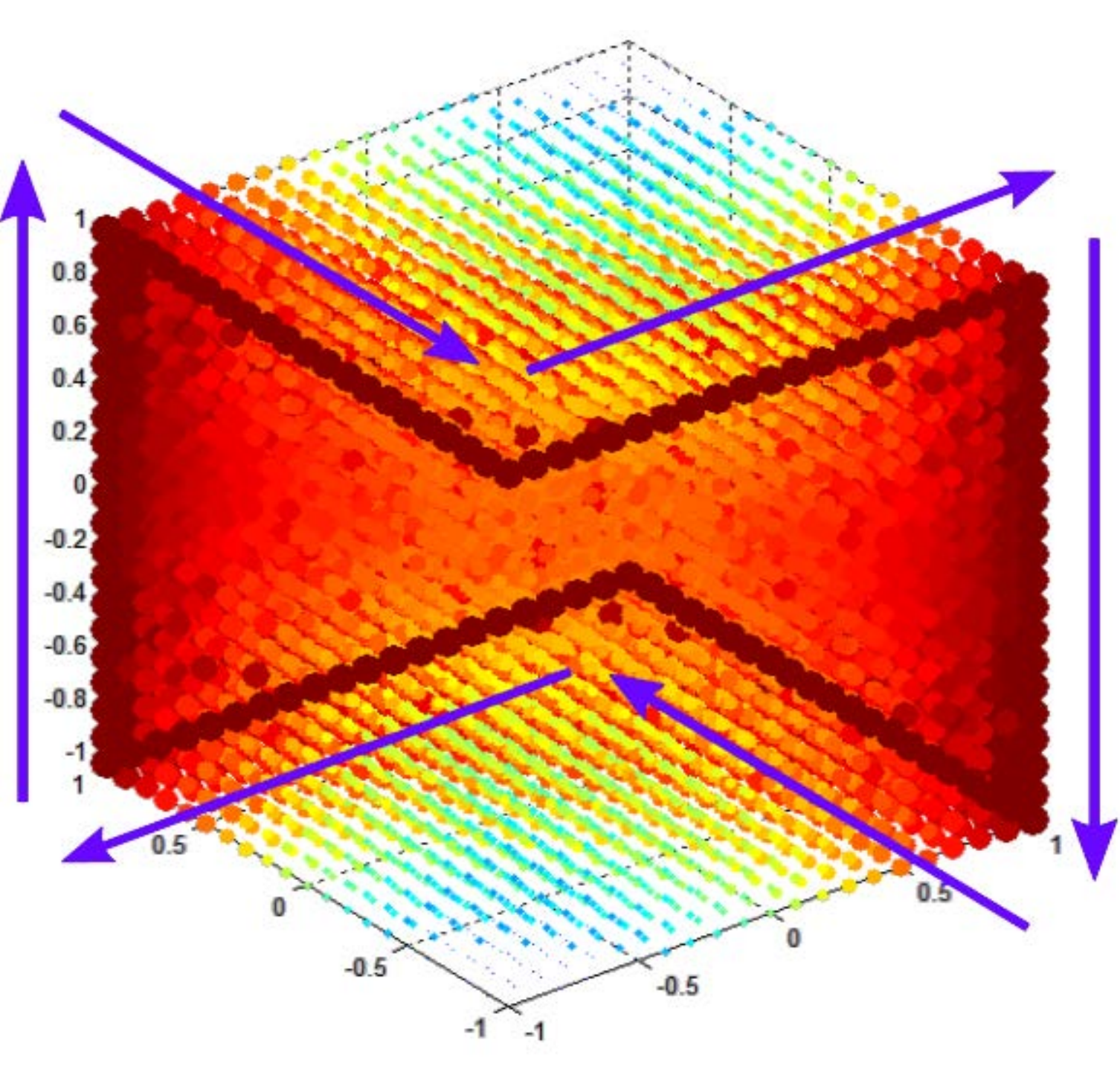}
\caption{Example of Powell~\cite{Pow73b} showing nonconvergence of
  cyclic coordinate descent.\label{fig:powell}}
\end{figure}

\subsection{Assumptions and Notation} \label{sec:assumptions}

For most of this section, we focus on the unconstrained problem
\eqref{eq:f}, where the objective $f$ is {\em convex} and Lipschitz
continuously differentiable. In some places, we assume strong
convexity with respect to the Euclidean norm, that is, existence of a
modulus of convexity $\sigma>0$ such that
\begin{equation} \label{eq:strong.convex}
f(y) \ge f(x) + \nabla f(x)^T(y-x) + \frac{\sigma}{2} \| y-x\|_2^2,
\quad \mbox{for all $x,y$.}
\end{equation}
(Henceforth, we use $\| \cdot \|$ to denote the Euclidean norm $\|
\cdot \|_2$, unless otherwise specified.)
We define Lipschitz constants that are tied to the component
directions, and are key to the algorithms and their analysis. The
first set of such constants are the {\em component Lipschitz
  constants}, which are positive quantities $L_i$ such that for all $x
\in \R^n$ and all $t \in \R$ we have
\begin{equation} \label{eq:Li}
| [\nabla f(x+te_i)]_i - [\nabla f(x)]_i| \le L_i |t|,
\end{equation}
We define the {\em coordinate Lipschitz constant} $\Lmax$ to be such that 
\begin{equation} \label{eq:Lmax}
\Lmax = \max_{i=1,2,\dotsc,n} \, L_i.
\end{equation}
The standard Lipschitz constant $L$ is such that 
\begin{equation} \label{eq:L}
\| \nabla f(x+d) - \nabla f(x) \| \le L \|d\|,
\end{equation}
for all $x$ and $d$ of interest. By referring to relationships between
norm and trace of a symmetric matrix, we can assume that
$1 \le L/\Lmax \le n$. (The upper bound is achieved when $f(x) = e(e^Tx)$, for $e=(1,1,\dotsc,1)^T$.)
 We also define the {\em restricted Lipschitz constant} $\Lres$ such
 that the following property is true for all $x \in \R^n$, all $t \in
 \R$, and all $i=1,2,\dotsc,n$:
\begin{equation} \label{eq:Lres}
\| \nabla f(x+te_i) - \nabla f(x) \| \le \Lres |t|.
\end{equation}
Clearly, $\Lres \le L$.
%% \begin{subequations} \label{eq:L}
%% \begin{align}
%% \| \nabla f(x) - \nabla f(x+te_i) \|_2 & \le \Lres |t|; \\
%% \| \nabla f(x) - \nabla f(x+te_i) \|_{\infty} & \le \Lmax |t|.
%% \end{align}
%% \end{subequations}
%% (For smooth convex $f$, we can identify $\Lmax$ with
%% $\max_{i=1,2,\dotsc,n} L_i$ by noting that 
%% \begin{align*}
%% \Lmax & = \max_{i,j} \sup_x |
%%     [\nabla^2 f(x)]_{ij} | \le \max_{i,j} \sup_x ([\nabla^2
%%       f(x)]_{ii}[\nabla^2 f(x)]_{jj})^{1/2} \\
%% & \le \max_{i,j} (L_i
%%     L_j)^{1/2} = \max_i L_i.
%% \end{align*}
%% The reverse inequality is trivial.)  
The ratio
\begin{equation} \label{eq:def:Lambda}
\Lambda := \Lres/\Lmax
\end{equation}
is  important in our analysis of asynchronous parallel
algorithms in Section~\ref{sec:parallel}.  In the case of $f$ convex and
twice continuously differentiable, we have by positive
semidefiniteness of the $\nabla^2 f(x)$ at all $x$ that
\[
| [\nabla^2 f(x)]_{ij} | \le \left( [\nabla^2 f(x)]_{ii}[\nabla^2 f(x)]_{jj} 
\right)^{1/2},
\]
from which we can deduce that 
\[
1 \le \Lambda \le \sqrt{n}.
\]
However, we can derive stronger bounds on $\Lambda$ for functions $f$
in which the coupling between components of $x$ is weak. In the
extreme case in which $f$ is separable, we have $\Lambda=1$.  The
coordinate Lipschitz constant corresponds $\Lmax$ to the maximal
absolute value of the diagonal elements of the Hessian $\nabla^2
f(x)$, while the restricted Lipschitz constant $\Lres$ is related to
the maximal column norm of the Hessian. Therefore, if the Hessian is
positive semidefinite and diagonally dominant, the ratio $\Lambda$ is
at most $2$.
%% When the Hessian
%% $\nabla^2 f(x)$ is diagonally dominant for all $x$, the coordinate
%% Lipschitz constant is the maximum diagonal, while the restricted
%% Lipschitz constant is the maximum column norm, and the ratio $\Lambda$
%% between them is at most $2$.

The following assumption is useful in the remainder of the paper.
\begin{assumption} \label{ass:fconv}
The function $f$ in \eqref{eq:f} is convex and uniformly Lipschitz
continuously differentiable, and attains its minimum value $f^*$ on a
set $\cS$. There is a finite $R_0$ such that the level set for $f$
defined by $x^0$ is bounded, that is,
\begin{equation} \label{eq:R0}
\max_{x^* \in \cS} \, \max_x \, \{ \| x-x^*\| \, : \, f(x) \le f(x^0) \} \le R_0.
\end{equation}
\end{assumption}

\subsection{Randomized Algorithms}
\label{sec:random}

In randomized CD algorithms, the update component $i_k$ is chosen
randomly at each iteration. In Algorithm~\ref{alg:rcd} we consider the
simplest variant in which each $i_k$ is selected from
$\{1,2,\dotsc,n\}$ with equal probability, independently of the
selections made at previous iterations. (We can think of this scheme
as ``sampling with replacement'' from the set $\{1,2,\dotsc,n\}$.)

\begin{algorithm} 
\caption{Randomized Coordinate Descent for \eqref{eq:f}\label{alg:rcd}}
\begin{algorithmic}
\State Choose $x^0 \in \R^n$;
\State Set $k \leftarrow 0$;
\Repeat
\State Choose index $i_k$ with uniform probability from
$\{1,2,\dotsc,n\}$, independently of choices at prior iterations;
% and set $d^k = [\nabla f(x^k)]_{i_k} e_{i_k}$; 
\State Set $x^{k+1} \leftarrow x^k - \alpha_k [\nabla f(x^k)]_{i_k} e_{i_k}$ for some $\alpha_k>0$; 
\State $k \leftarrow k+1$;
\Until termination test satisfied;
\end{algorithmic}
\end{algorithm}

We denote expectation with respect to a single random index $i_k$ by
$E_{i_k} (\cdot)$, while $E(\cdot)$ denotes expectation with respect
to all random variables $i_0,i_1,i_2,\dotsc$.

%% We define $\xi_k= \{i_0,i_1,\dotsc,i_k\}$ to be the sequence of
%% components updated at all iterations up to and including iteration
%% $k$.  Expectation with respect to $\xi_k$ and $i_k$ is denoted by
%% $E_{\xi_k}$ and $E_{i_k}$, respectively, with the subscript omitted
%% when the context is clear.

We prove a convergence result for the randomized algorithm, for the
simple steplength choice $\alpha_k \equiv 1/\Lmax$. (The proof is a
simplified version of the analysis in
Nesterov~\cite[Section~2]{Nes10a}. A result similar to \eqref{eq:1k}
is proved by Shalev-Schwartz and Tewari~\cite{ShaT11a} for certain
types of $\ell_1$-regularized problems.)
\begin{theorem} \label{th:rcd}
Suppose that Assumption~\ref{ass:fconv} holds. Suppose that $\alpha_k
\equiv 1/\Lmax$ in Algorithm~\ref{alg:rcd}. Then for all $k>0$ we have
\begin{equation} \label{eq:1k}
E(f(x^k)) - f^* \le \frac{2n\Lmax R_0^2}{k}.
\end{equation}
When $\sigma>0$ in \eqref{eq:strong.convex}, we have in addition that
\begin{equation} \label{eq:rcd.linear}
E \left( f(x^k) \right) - f^* \le \left( 1-\frac{\sigma}{n \Lmax}
\right)^k (f(x^0)-f^*).
\end{equation}
\end{theorem}
\begin{proof}
By application of Taylor's theorem, and using \eqref{eq:Li} and
\eqref{eq:Lmax}, we have
\begin{align}
\nonumber
f(x^{k+1}) &= f \left( x^k - \alpha_k [\nabla f(x^k)]_{i_k} e_{i_k} \right) \\
\nonumber
& \le f(x^k) - \alpha_k [\nabla f(x^k)]_{i_k}^2 + \frac12 \alpha_k^2 L_{i_k}
[\nabla f(x^k)]_{i_k}^2 \\
 \nonumber
& \le f(x^k) - \alpha_k \left( 1-\frac{\Lmax}{2} \alpha_k \right) [\nabla f(x^k)]_{i_k}^2 \\
 \label{eq:rcd.1}
& = f(x^k) - \frac{1}{2 \Lmax}  [\nabla f(x^k)]_{i_k}^2,
\end{align}
where we substituted the choice $\alpha_k = 1/\Lmax$ in the last
equality. Taking the expectation of both sides of this expression
over the random index $i_k$, we have
\begin{align}
\nonumber
E_{i_k} f(x^{k+1})  & \le f(x^k) - \frac{1}{2 \Lmax} \frac{1}{n} \sum_{i=1}^m [\nabla f(x^k)]_{i}^2  \\
\label{eq:rcd.2}
& =
f(x^k) - \frac{1}{2n\Lmax} \| \nabla f(x^k) \|^2.
\end{align}
(We used here the facts that $x^k$ does not depend on $i_k$, and that
$i_k$ was chosen from among $\{1,2,\dotsc,n\}$ with equal
probability.)  We now subtract $f(x^*)$ from both sides this
expression, take expectation of both sides with respect to {\em all}
random variables $i_0,i_1, \dotsc$, and use the notation
\begin{equation} \label{eq:phik}
\phi_k := E (f(x^k)) - f^*.
\end{equation}
to obtain
\begin{equation}
\phi_{k+1} \le \phi_k -  \frac{1}{2n\Lmax} E \left( \| \nabla f(x^k) \|^2 \right)  
\label{eq:rcd.3}
 \le \phi_k -  \frac{1}{2n \Lmax}  \left[ E ( \| \nabla f(x^k) \| ) \right]^2.
\end{equation}
(We used Jensen's Inequality in the second inequality.) By convexity
of $f$ we have for any $x^* \in \cS$ that
\[
f(x^k) - f^* \le \nabla f(x^k)^T(x^k-x^*) \le \| \nabla f(x^k) \| \| x^k - x^* \| \le R_0 \| \nabla f(x^k) \|,
\]
where the final inequality is because $f(x^k) \le f(x^0)$, so that
$x^k$ is in the level set in \eqref{eq:R0}. 
By taking expectations of both sides, we obtain
\[
E ( \| \nabla f(x^k) \| ) \ge \frac{1}{R_0} \phi_k.
\]
When we substitute this bound into \eqref{eq:rcd.3}, and rearrange, we
obtain
\[
\phi_k - \phi_{k+1} \ge \frac{1}{2n\Lmax}  \frac{1}{R_0^2} \phi_k^2.
\]
We thus have
\[
\frac{1}{\phi_{k+1}} - \frac{1}{\phi_k} = \frac{\phi_k-\phi_{k+1}}{\phi_k \phi_{k+1}} \ge \frac{\phi_k-\phi_{k+1}}{\phi_k^2}  \ge \frac{1}{2n\Lmax R_0^2}.
\]
By applying this formula recursively, we obtain
\[
\frac{1}{\phi_k} \ge \frac{1}{\phi_0} + \frac{k}{2n\Lmax R_0^2} \ge
\frac{k}{2n\Lmax R_0^2},
\]
so that \eqref{eq:1k} holds, as claimed.

In the case of $f$ strongly convex with modulus $\sigma>0$, we have by
taking the minimum of both sides with respect to $y$ in
\eqref{eq:strong.convex}, and setting $x=x^k$, that
\[
f^* \ge f(x^k) - \frac{1}{2 \sigma} \| \nabla f(x^k) \|^2.
\]
By using this expression to bound $\| \nabla f(x^k) \|^2$ in
\eqref{eq:rcd.3}, we obtain
\[
\phi_{k+1} \le \phi_k - \frac{\sigma}{n \Lmax} \phi_k = 
\left( 1-\frac{\sigma}{n \Lmax}
\right) \phi_k.
\]
Recursive application of this formula leads to \eqref{eq:rcd.linear}.
\end{proof}

Note that the same convergence expressions can be obtained for more
refined choices of steplength $\alpha_k$, by making minor adjustments
to the logic in \eqref{eq:rcd.1}. For example, the choice $\alpha_k =
1/L_{i_k}$ leads to the same bounds \eqref{eq:1k} and
\eqref{eq:rcd.linear}. The same bounds hold too when $\alpha_k$ is the
exact minimizer of $f$ along the coordinate search direction; we
modify the logic in \eqref{eq:rcd.1} for this case by taking the
minimum of all expressions with respect to $\alpha_k$, and use the
fact that $\alpha_k = 1/\Lmax$ is in general a suboptimal choice.

We can compare \eqref{eq:1k} with the corresponding result for
full-gradient descent with constant steplength $\alpha_k = 1/L$ (where
$L$ is from \eqref{eq:L}). The iteration 
\[
x^{k+1} = x^k - \frac{1}{L} \nabla f(x^k)
\]
leads to a convergence expression
\begin{equation} \label{eq:ssd}
f(x^k) - f^* \le \frac{2L R_0^2}{k}
\end{equation}
(see, for example, \cite{Nes04}). Since, as we have noted, $L$ can be
as large as $n \Lmax$, the bound in this expression may be equivalent
to \eqref{eq:1k} in extreme cases. More typically, these two Lipschitz
constants are comparable in size, and the appearance of the additional
factor $n$ in \eqref{eq:1k} indicates that we pay a price in terms of
slower convergence for using only one component of $\nabla f(x^k)$,
rather than the full vector.

Expected linear convergence rates have been proved under assumptions
weaker than strong convexity; see for example the ``essential strong
convexity'' property of \cite{LiuW13a}, the ``optimal strong
convexity'' property of \cite{LiuW14c}, the ``generalized error
bound'' property of \cite{NecC13a}, and \cite[Assumption~2]{TseY06},
which concerns linear growth in a measure of the gradient with
distance from the solution set.

A variant on Algorithm~\ref{alg:rcd} uses ``sampling without
replacement.'' Here the computation proceeds in ``epochs'' of $n$
consecutive iterations. At the start of each epoch, the set
$\{1,2,\dotsc, n\}$ is shuffled. The iterations then proceed
by setting $i_k$ to each entry in turn from the ordered set. This kind
of randomization has been shown in several contexts to be superior to
the sampling-with-replacement scheme analyzed above, but a
theoretical understanding of this phenomenon remains elusive.

\paragraph{Randomized Kaczmarz Algorithm.}

It is worth proving an expected linear convergence result for the
Kaczmarz iteration \eqref{eq:kac.w} for linear equations $Aw=b$ as a
separate, more elementary analysis. In one sense, the result is a
special case of Theorem~\ref{th:rcd} since, as we showed above, the
iteration \eqref{eq:kac.w} is obtained by applying
Algorithm~\ref{alg:rcd} to the dual formulation \eqref{eq:kac}. In
another sense, the result is stronger, since we obtain a linear rate
of convergence without requiring strong convexity of the objective
\eqref{eq:kac}, that is, the system $Aw=b$ is allowed to have multiple
solutions.

We denote by $\lambdaminnz$ the minimum nonzero eigenvalue of $AA^T$
and let $P(\cdot)$ denote projection onto the solution set of $Aw=b$.  We
have
\begin{align*}
\|w^{k+1}-P(w^{k+1})\|^2 & \le 
\| w^k - A_{i_k}^T (A_{i_k}w^k-b_{i_k}) - P(w^k) \|^2 \\
&= \frac12
\|w^k - P(w^k) \|^2 -  (A_{i_k} w^k - b_{i_k} )^2,
\end{align*}
where we have used normalization of the rows \eqref{eq:Anormalized}
and the fact that $A_{i_k} P(x^k) = b_{i_k}$.  By taking expectations
of both sides with respect to $i_k$, we have
\begin{align*}
E_{i_k}  \|w^{k+1}-P(w^{k+1})\|^2
&\le 
\|w^k - P(w^k) \|^2 -  E_{i_k} (A_{i_k} w^k - b_{i_k} )^2 \\
&= \frac12
\|w^k - P(w^k) \|^2 - \frac{1}{m} \|Aw^k -b \|^2  \\
& \le \left( 1- \frac{\lambdaminnz}{m} \right) \|w^k-P(w^k) \|^2.
\end{align*}
By taking expectations of both sides with respect to all random
variables $i_0,i_1, \dotsc$, and proceeding recursively, we obtain
\[
E \| w^k - P(w^k) \|^2  \le
 \left( 1- \frac{\lambdaminnz}{m} \right)^k  \| w^0-P(w^0) \|^2.
\]
(This analysis is slightly generalized from Strohmer and
Vershynin~\cite{Strohmer09} to allow for nonunique solutions of
$Aw=b$; see also \cite{LevL10a}.)

\subsection{Accelerated Randomized Algorithms}
\label{sec:accel}

The accelerated randomized algorithm, specified here as
Algorithm~\ref{alg:arcd}, was proposed by Nesterov~\cite{Nes10a}. It
assumes that an estimate is available of modulus of strong convexity
$\sigma \ge 0$ from \eqref{eq:strong.convex}, as well as estimates of
the component-wise Lipschitz constants $L_i$ from \eqref{eq:Li}. (The
algorithm remains valid if we simply use $\Lmax$ in place of
$L_{i_k}$ for all $k$.)

\begin{algorithm} 
\caption{Accelerated Randomized Coordinate Descent for \eqref{eq:f}\label{alg:arcd}}
\begin{algorithmic}
\State Choose $x^0 \in \R^n$;
\State Set $k \leftarrow 0$, $v^0 \leftarrow x^0$, $\gamma_{-1}\leftarrow 0$;
\Repeat
\State Choose $\gamma_k$ to be the larger root of 
\[
\gamma_k^2 - \frac{\gamma_k}{n} = \left( 1-\frac{\gamma_k \sigma}{n} \right)
\gamma_{k-1}^2.
\]
\State Set
\begin{equation} \label{eq:ab}
\alpha_k \leftarrow \frac{n-\gamma_k \sigma}{\gamma_k(n^2-\sigma)}, \quad
\beta_k \leftarrow 1-\frac{\gamma_k \sigma}{n};
\end{equation}
\State Set $y^k \leftarrow \alpha_k v^k + (1-\alpha_k) x^k$; 
\State Choose index $i_k \in \{1,2,\dotsc,n\}$ with uniform probability
and set $d^k = [\nabla f(y^k)]_{i_k} e_{i_k}$; 
\State Set $x^{k+1} \leftarrow y^k - (1/L_{i_k}) d^k$; 
\State Set $v^{k+1} \leftarrow  \beta_k v^k + (1-\beta^k) y^k - (\gamma_k/L_{i_k}) d^k$; 
\State $k \leftarrow k+1$;
\Until termination test satisfied;
\end{algorithmic}
\end{algorithm}

The approach is a close relative of the accelerated (full-)gradient
methods that have become extremely popular in recent years. These
methods have their origin in a 1983 paper of Nesterov~\cite{Nes83} and
owe much of their recent popularity to a recent incarnation known as
FISTA \cite{BecT08a} and an exposition in Nesterov's 2004
monograph~\cite{Nes04}, as well as ease of implementation and good
practical experience. In their use of momentum in the choice of step
--- the search direction combines new gradient information with the
previous search direction --- these methods are also related to such
other classical techniques as the heavy-ball method (see \cite{Pol87})
and conjugate gradient methods.

Nesterov~\cite[Theorem~6]{Nes10a} proves the following convergence
result for Algorithm~\ref{alg:arcd}.
\begin{theorem} \label{th:arcd}
Suppose that Assumption~\ref{ass:fconv} holds, and define 
\[
S_0 := \sup_{x^* \in \cS} \, \Lmax \| x^0-x^* \|^2 + (f(x^0)-f^*)/n^2.
\]
Then for all $k \ge 0$ we have
\begin{align}
\nonumber
E& (f(x^k))  - f^* \\
\label{eq:arcd.1}
 & \le S_0 \frac{\sigma}{\Lmax}  \left[ \left( 1+ \frac{\sqrt{\sigma/\Lmax}}{2n} \right)^{k+1} -
\left( 1- \frac{\sqrt{\sigma/\Lmax}}{2n} \right)^{k+1} \right]^{-2} \\
\label{eq:arcd.2}
& \le S_0 \left( \frac{n}{k+1} \right)^2.
\end{align}
\end{theorem}

In the strongly convex case $\sigma>0$, the term $(1+
\sqrt{\sigma/\Lmax}/(2n))^{k+1}$ eventually dominates the second term
in brackets in \eqref{eq:arcd.1}, so that the linear convergence rate
suggested by this expression is significantly faster than the
corresponding rate \eqref{eq:rcd.linear} for
Algorithm~\ref{alg:rcd}. Essentially, the measure $\sigma/\Lmax$ of
conditioning in \eqref{eq:rcd.linear} is replaced by its square root
in \eqref{eq:arcd.1}, suggesting a decrease by a factor of
$\sqrt{\Lmax/\sigma}$ in the number of iterations required to meet a
specified error tolerance. In the sublinear rate bound
\eqref{eq:arcd.2}, which holds even for weakly convex $f$, the $1/k$
bound of \eqref{eq:1k} is replaced by a $1/k^2$ factor, implying a
reduction from $O(1/\epsilon)$ to $O(1/\sqrt{\epsilon})$ in the number
of iterations required to meet a specified error tolerance.

\subsection{Efficient Implementation of the Accelerated Algorithm}

\begin{algorithm} 
\caption{Accelerated Randomized Kaczmarz for \eqref{eq:Aw=b}, \eqref{eq:Anormalized}\label{alg:ark}}
\begin{algorithmic}
\State Choose $w^0 \in \R^n$;
\State Set $k \leftarrow 0$, $\tilde{v}^0 \leftarrow w^0$, $\gamma_{-1}\leftarrow 0$;
\Repeat
\State Choose $\gamma_k$ to be the larger root of 
\[
\gamma_k^2 - \frac{\gamma_k}{n} = \left( 1-\frac{\gamma_k \sigma}{n} \right)
\gamma_{k-1}^2.
\]
\State Set
\begin{equation} \label{eq:ab.ark}
\alpha_k \leftarrow \frac{n-\gamma_k \sigma}{\gamma_k(n^2-\sigma)}, \quad
\beta_k \leftarrow 1-\frac{\gamma_k \sigma}{n};
\end{equation}
\State Set $\tilde{y}^k \leftarrow \alpha_k \tilde{v}^k + (1-\alpha_k) w^k$; 
\State Choose index $i_k \in \{1,2,\dotsc,m\}$ with uniform probability
and set $\tilde{d}^k =  (A_{i_k} \tilde{y}^k - b_{i_k}) A_{i_k}^T$; 
\State Set $w^{k+1} \leftarrow \tilde{y}^k - \tilde{d}^k$; 
\State Set $\tilde{v}^{k+1} \leftarrow  \beta_k \tilde{v}^k + (1-\beta^k) \tilde{y}^k - \gamma_k \tilde{d}^k$; 
\State $k \leftarrow k+1$;
\Until termination test satisfied;
\end{algorithmic}
\end{algorithm}

One fact detracts from the appeal of accelerated CD methods over
standard methods: the higher cost of each iteration of
Algorithm~\ref{alg:arcd}. Both standard and accelerated variants
require calculation of one element of the gradient, but
Algorithm~\ref{alg:rcd} requires an update of just a single component
of $x$, whereas Algorithm~\ref{alg:arcd} also requires manipulation of
the generally dense vectors $y$ and $v$.  Moreover, the gradient is
evaluated at $x^k$ in Algorithm~\ref{alg:rcd}, where the argument
changes by only one component from the prior iteration, a fact that
can be exploited in several contexts. In Algorithm~\ref{alg:arcd}, the
argument $y^k$ for the gradient changes more extensively from one
iteration to the next, making it less obvious whether such economies
are available.  However, by using a change of variables due to Lee and
Sidford~\cite{LeeS13}, it is possible to implement the accelerated
randomized CD approach efficiently for problems with certain
structure, including the linear system $Aw=b$ and certain problems of
the form \eqref{eq:erm}.

We explain the Lee-Sidford technique in the context of the Kaczmarz
algorithm for \eqref{eq:Aw=b}, assuming normalization of the rows of
$A$ \eqref{eq:Anormalized}. As we explained in \eqref{eq:kac.w}, the
Kaczmarz algorithm is obtained by applying CD to the dual formulation
\eqref{eq:kac} with variables $x$, but operating in the space of
``primal'' variables $w$ using the transformation $w=A^Tx$. If we
apply the transformations $\tilde{v}^k = A^T v^k$ and $\tilde{y}^k =
A^T y^k$ to the other vectors in Algorithm~\ref{alg:arcd}, and use the
fact of normalization \eqref{eq:Anormalized} (and hence
$(AA^T)_{ii}=1$ for all $i=1,2,\dotsc,m$) to note that $L_i \equiv 1$
in \eqref{eq:Li}, we obtain Algorithm~\ref{alg:ark}.

When the matrix $A$ is dense, there is only a small factor of
difference between the per-iteration workload of the standard Kaczmarz
algorithm and its accelerated variant, Algorithm~\ref{alg:ark}. Both
require $O(m+n)$ operations per iteration. However, when $A$ is
sparse, the computational difference between the two algorithms
becomes substantial. At iteration $k$, the standard Kaczmarz algorithm
requires computation proportion to a small multiple of the number of
nonzeros in row $A_{i_k}$ (which we denote by $|A_{i_k}|$).
Meanwhile, iteration $k$ of Algorithm~\ref{alg:ark} requires
manipulation of the dense vectors $\tilde{v}^k$ and $\tilde{y}^k$ ---
both $O(n)$ processes --- and the benefits of sparsity are lost. This
apparent defect was partly remedied in \cite{LiuW13d} by ``caching''
the updates to these vectors, resulting in a number of cycles within
which updates gradually ``fill in.'' The more effective approach of
\cite{LeeS13} performs a change of variables from $\tilde{v}^k$ and
$\tilde{y}^k$ to two other vectors $\hat{v}^k$ and $\hat{y}^k$ that
     {\em can} be updated in $O(| A_{i_k}|)$ operations.  To describe
     this representation, we start by noting that if we substitute for
     $w^k$ and $w^{k+1}$ in the formulas of Algorithm~\ref{alg:ark},
     we obtain the updates to $\tilde{v}^k$ and $\tilde{y}^k$ in the
     following form:
\begin{equation} \label{eq:lees.2}
\left[ \begin{matrix} \tilde{v}^{k+1} & \tilde{y}^{k+1} \end{matrix} \right] =
\left[ \begin{matrix} \tilde{v}^{k} & \tilde{y}^{k} \end{matrix} \right] R_k 
- S_k,
\end{equation}
where
\begin{align*}
R_k & := \left[ \begin{matrix} \beta_k & \alpha_{k+1} \beta_k \\
(1-\beta_k)  & (1-\alpha_{k+1} \beta_k) \end{matrix} \right], \\
S_k & := (A_{i_k} \tilde{y}^k - b_{i_k}) A_{i_k}^T
\left[ \begin{matrix} \gamma_k & (1-\alpha_{k+1} + \alpha_{k+1} \gamma_k)
 \end{matrix} \right].
\end{align*}
Note that $R_k$ is a $2 \times 2$ matrix while $S_k$ is an $n \times
2$ matrix with nonzeros only in those rows for which $A_{i_k}^T$ has a
nonzero entry.  We define a change of variables based on another $2
\times 2$ matrix $B_k$, as follows:
\begin{equation} \label{eq:lees.1}
\left[ \begin{matrix} \tilde{v}^k & \tilde{y}^k \end{matrix} \right] =
\left[ \begin{matrix} \hat{v}^k & \hat{y}^k \end{matrix} \right] B_k,
\end{equation}
where we initialize with $B_0=I$. By substituting this representation
into \eqref{eq:lees.2}, we obtain
\[
\left[ \begin{matrix} \hat{v}^{k+1} & \hat{y}^{k+1} \end{matrix} \right] B_{k+1} =
\left[ \begin{matrix} \hat{v}^{k} & \hat{y}^{k} \end{matrix} \right] B_k R_k 
- S_k,
\]
so we can maintain validity of the representation \eqref{eq:lees.1} at
iteration $k+1$ by setting
\begin{equation} \label{eq:lees.3}
B_{k+1} := B_k R_k, \quad
\left[ \begin{matrix} \hat{v}^{k+1} & \hat{y}^{k+1} \end{matrix} \right]
:=
\left[ \begin{matrix} \hat{v}^{k} & \hat{y}^{k} \end{matrix} \right] - S_k B_{k+1}^{-1}.
\end{equation}
The computations in \eqref{eq:lees.3} can be performed in
$O(|A_{i_k}|)$ operations, and can replace the relatively expensive
computations of $\tilde{y}^k$ and $\tilde{v}^{k+1}$ in
Algorithm~\ref{alg:ark}. The only other operation of note in this
algorithm --- computation of $A_{i_k} \tilde{y}^k - b_{i_k}$ --- can
also be performed in $O(|A_{i_k}|)$ operations using the
$(\hat{v}^k,\hat{y}^k)$ representation, by noting from
\eqref{eq:lees.1} that
\[
A_{i_k} \tilde{y}^k = (A_{i_k} \hat{v}^k) (B_k)_{12} + (A_{i_k} \hat{y}^k) (B_k)_{22}.
\]

This efficient implementation can be extended to the dual empirical
risk minimization problem \eqref{eq:erm.dual} for certain choices of
regularization function $g(\cdot)$, for example, $g(z) = \|z\|^2/2$;
see \cite{LinLX14a}. As pointed out in \cite{LeeS13}, the key
requirement for the efficient scheme is that the gradient term
$[\nabla f(y^k)]_{i_k}$ can be evaluated efficiently after an update
to the two vectors in the alternative representation of $y^k$, and to
the two coefficients in this representation. Another variant of this
implementation technique appears in \cite[Section~5]{FerR13b}.

\subsection{Cyclic Variants}
\label{sec:cyclic}

We have the following result from \cite{BecT12a} for the cyclic variant
of Algorithm~\ref{alg:cd}.
\begin{theorem} \label{th:cd.cyclic}
Suppose that Assumption~\ref{ass:fconv} holds. Suppose that $\alpha_k
\equiv 1/\Lmax$ in Algorithm~\ref{alg:cd}, with the index $i_k$ at
iteration $k$ chosen according to the cyclic ordering \eqref{eq:cyc}
(with $i_0=1$). Then for $k=n,2n,3n, \dotsc$, we have
\begin{equation} \label{eq:cyc.1k}
f(x^{k}) - f^* \le \frac{4n \Lmax (1+n L^2 / \Lmax^2) R_0^2}{k+8}.
\end{equation}
When $\sigma>0$ in the strong convexity condition
\eqref{eq:strong.convex}, we have in addition for $k=n,2n,3n, \dotsc$
that
\begin{equation} \label{eq:cyc.linear}
 f(x^k)  - f^* \le \left( 1-\frac{\sigma}{2\Lmax(1+n L^2/\Lmax^2)}
\right)^{k/n} (f(x^0)-f^*).
\end{equation}
\end{theorem}
\begin{proof}
The result \eqref{eq:cyc.1k} follows from Theorems~3.6 and 3.9 in
\cite{BecT12a} when we note that (i) each iteration of Algorithm BCGD
in \cite{BecT12a} corresponds to a ``cycle'' of $n$ iterations in
Algorithm~\ref{alg:cd}; (ii) we update coordinates rather than blocks,
so that the parameter $p$ in \cite{BecT12a} is equal to $n$; (iii) we
set $\bar{L}_{\max}$ and $\bar{L}_{\min}$ in \cite{BecT12a} both to
$\Lmax$.
\end{proof}

Comparing the complexity bounds for the cyclic variant with the
corresponding bounds proved in Theorem~\ref{th:rcd} for the randomized
variant, we see that since $L \ge \Lmax$ in general, the numerator in
\eqref{eq:cyc.1k} is $O(n^2)$, in contrast to $O(n)$ term in
\eqref{eq:1k}.  A similar factor of $n$ in seen in comparing
\eqref{eq:rcd.linear} to \eqref{eq:cyc.linear}, when we note that
$(1-\epsilon)^{1/n} \approx 1-\epsilon/n$ for small values of
$\epsilon$. The bounds in Theorem~\ref{th:cd.cyclic} are
deterministic, however, rather than being bounds on expected
nonoptimality, as in Theorem~\ref{th:rcd}.

We noted in Subsection~\ref{sec:assumptions} that the ratio $L/\Lmax$
lies in the interval $[1,n]$ when $f$ is a convex quadratic function
and both parameters are set to their best values.  Lower values of
this ratio are attained on functions that are ``more decoupled'' and
larger values attained when there is a greater dependence between the
coordinates. Larger values lead to weaker bounds in
Theorem~\ref{th:cd.cyclic}, which accords with our intuition; we
expect CD methods to require more iterations to resolve the coupling
of the coordinates. 

We are free to make other, larger choices of $\Lmax$; they need only
satisfy the conditions \eqref{eq:Li} and \eqref{eq:Lmax}. Larger
values of $\Lmax$ lead to shorter steps $\alpha_k=1/\Lmax$ and
different complexity expressions. For $\Lmax=L$, for example, the
bound in \eqref{eq:cyc.1k} becomes
\[
\frac{4n(n+1) L R_0^2}{k+8},
\]
which is worse by a factor of approximately $2n^2$ than the bound
\eqref{eq:ssd} for the full-step gradient descent approach. For
$\Lmax=\sqrt{n}L$, we obtain
\[
\frac{8n^{3/2} L R_0^2}{k+8},
\]
which still trails \eqref{eq:ssd} by a factor of $4n^{3/2}$.

\subsection{Extension to Separable Regularized Case} \label{sec:sepreg}

In this section we consider the separable regularized formulation
\eqref{eq:reg}, \eqref{eq:reg.sep} where $f$ is smooth and strongly
convex, and each $\Omega_i$, $i=1,2,\dotsc,n$ is convex. We prove a
result similar to the second part of Theorem~\ref{th:rcd} for a
randomized version of Algorithm~\ref{alg:cdreg}. The proof is a
simplified version of the analysis from \cite{RicT11a}. It makes use
of the following assumption.

\begin{assumption} \label{ass:hconv}
The function $f$ in \eqref{eq:reg} is uniformly Lipschitz continuously
differentiable and strongly convex with modulus $\sigma>0$ (see
\eqref{eq:strong.convex}). The functions $\Omega_i$,
$i=1,2,\dotsc,n$ are convex. The function $h$ in \eqref{eq:reg} attains its
minimum value $h^*$ at a unique point $x^*$.
\end{assumption} 

Our result uses the coordinate Lipschitz constant $\Lmax$ for $f$, as
defined in \eqref{eq:Lmax}. Note that the modulus of convexity
$\sigma$ for $f$ is also the modulus of convexity for $h$. By
elementary results for convex functions, we have
\begin{equation} \label{eq:h.strongly.convex}
h(\alpha x + (1-\alpha)y) \le \alpha h(x) + (1-\alpha) h(y) - \frac12
\sigma \alpha (1-\alpha) \|x-y \|^2.
\end{equation}

\begin{theorem} \label{th:rcd.reg} 
Suppose that Assumption~\ref{ass:hconv} holds. Suppose that the
indices $i_k$ in Algorithm~\ref{alg:cdreg} are chosen independently
for each $k$ with uniform probability from $\{1,2,\dotsc,n\}$, and
that $\alpha_k \equiv 1/\Lmax$. Then for all $k \ge 0$, we have
\begin{equation} \label{eq:rcd.reg.linear}
E \left( h(x^k) \right) - h^* \le \left( 1-\frac{\sigma}{n \Lmax}
\right)^k (h(x^0)-h^*).
\end{equation}
\end{theorem}
\begin{proof}
Define the function
\[
H(x^k,z) := f(x^k) + \nabla f(x^k)^T(z-x^k) + \frac12 \Lmax \| z-x^k\|^2 +
\lambda \Omega(z),
\]
and note that this function is separable in the components of $z$, and
attains its minimum over $z$ at the vector $z^k$ whose $i_k$ component
is defined in Algorithm~\ref{alg:cdreg}. 
Note by  strong convexity \eqref{eq:strong.convex} that
\begin{align} 
\nonumber
H(x^k,z) & \le f(z) - \frac12 \sigma \|z-x^k\|^2 + \frac12 \Lmax \| z-x^k\|^2 + \lambda \Omega(z) \\
\label{eq:zig.1}
& = h(z) + \frac12 (\Lmax-\sigma) \| z-x^k\|^2.
\end{align}
We have by minimizing both sides over $z$ in this expression that
\begin{align}
\nonumber
H(x^k,z^k) &= \min_z \, H(x^k,z) \\
\nonumber
& \le \min_z \, h(z) + \frac12 (\Lmax-\sigma) \| z-x^k\|^2                        \\
\nonumber
 & \le \min_{\alpha \in [0,1]} \, h(\alpha x^* + (1-\alpha) x^k) +
\frac12 (\Lmax-\sigma) \alpha^2 \|x^k-x^* \|^2                                    \\
\nonumber
 & \le \min_{\alpha \in [0,1]} \, \alpha h^* + (1-\alpha) h(x^k) + \frac12
\left[ (\Lmax-\sigma) \alpha^2 - \sigma \alpha (1-\alpha) \right] \| x^k-x^* \|^2 \\
\label{eq:zig.2}
 & \le \frac{\sigma}{\Lmax} h^* + \left( 1-\frac{\sigma}{\Lmax} \right) h(x^k),
\end{align}
where we used \eqref{eq:zig.1} for the first inequality,
\eqref{eq:h.strongly.convex} for the third inequality, and the
particular value $\alpha = \sigma/\Lmax$ for the fourth inequality
(for which value the coefficient of $\|x^k-x^* \|^2$ vanishes).
Taking the expected value of $h(x^{k+1})$ over the index $i_k$, we
have
\begin{align*}
E_{i_k} h(x^{k+1}) &= \frac{1}{n} \sum_{i=1}^n 
\left[ f(x^k+ (z^k_i - x^k_i)e_i) + \lambda \Omega_i(z^k_i) + \lambda
\sum_{j \neq i} \Omega_j(x^k_j) \right] \\
& \le \frac{1}{n} \sum_{i=1}^n \left\{ f(x^k) + [\nabla f(x^k)]_i (z^k_i-x^k_i)
+ \frac12 \Lmax (z^k_i-x^k_i)^2 \right. \\
& \quad\quad\quad\quad\quad\quad\quad\quad\quad\quad\quad\quad 
 \left. + \lambda \Omega_i(z^k_i) + \lambda \sum_{j \neq i} \Omega_j(x^k_j) \right\} \\
& = \frac{n-1}{n} h(x^k) + \frac{1}{n} \left[ f(x^k) + \nabla f(x^k)^T(z^k-x^k) \right. \\
& \quad\quad\quad\quad\quad\quad\quad\quad\quad\quad  +
\left. \frac12 \Lmax \| z^k-x^k\|^2 + \lambda \Omega(z^k) \right] \\
&= \frac{n-1}{n} h(x^k) + \frac{1}{n} H(x^k,z^k).
\end{align*}
By subtracting $h^*$ from both sides of this expression, and
using  \eqref{eq:zig.2} to substitute for $H(x^k,z^k)$, we obtain
\[
E_{i_k} h(x^{k+1}) - h^* \le \left( 1-\frac{\sigma}{n\Lmax} \right) (h(x^k)-h^*).
\]
By taking expectations of both sides of this expression with respect to the
random indices $i_0,i_1,i_2, \dotsc, i_{k-1}$, we obtain
\[
E( h(x^{k+1})) - h^* \le \left( 1-\frac{\sigma}{n\Lmax} \right) (E(h(x^k))-h^*).
\]
The result follows from a recursive application of this formula.
\end{proof}

A result similar to \eqref{eq:1k} can be proved for the case in which
$f$ is convex but not strongly convex, but there are a few technical
complications, and we refer the reader to \cite{RicT11a} for details.

An extension of the fixed-step approach to separable composite
objectives \eqref{eq:reg}, \eqref{eq:reg.sep} with {\em nonconvex}
smooth part $f$ is discussed in \cite{PatN13a}, where it is shown that
accumulation points of the sequence of iterates are stationary and
that a measure of optimality decreases to zero at a sublinear ($1/k$)
rate.

\subsection{Computational Notes}
\label{sec:computation}

A full computational comparison between variants of CD (and between CD
and other methods) is beyond the scope of this paper. Nevertheless it
is worth asking whether various aspects of the convergence analysis
presented above --- in particular, the distinction between CD variants
--- can be observed in practice. To this end, we used these methods to
minimize a convex quadratic $f(x)=(1/2) x^TQx$ (with $Q$ symmetric and
positive semidefinite) for which $x^*=0$ and $f^*=0$. We constructed $Q$
by choosing an integer $r$ from $1,2,\dotsc,n$ and parameters $\eta
\in [0,1]$ and $\zeta>0$, and defining
\begin{subequations}
\begin{align} \label{eq:Q}
Q & := V_{r,\eta} \Sigma V_{r,\eta}^T + \zeta \bfone \bfone^T, \\
V_{r,\eta} & := \eta V+(1-\eta) E_r, \\
 E_r &:= [I_{r \times r} \, | \, 0_{r \times (n-r)}]^T.
\end{align}
\end{subequations}
where $V \in \R^{n \times r}$ is a random matrix with $r \le n$
orthogonal columns, $\Sigma$ is an $r \times r$ positive diagonal
matrix whose diagonal elements were chosen from a log-uniform
distribution to have a specified condition number (with maximum
diagonal of $1$), and $\bfone$ is the vector $(1,1,\dotsc,1)^T$. For
convenience, we normalized $Q$ so that its maximum diagonal --- and
thus $\Lmax$ \eqref{eq:Lmax} --- is $1$.

By choosing $\eta$ and $\zeta$ appropriately, we can obtain a range of
values for the quantities described in
Subsection~\ref{sec:assumptions}, which enter along with the smallest
singular value into the convergence expression. For example, by
setting $\zeta=0$ and $\eta=0$ we obtain a randomly oriented matrix,
possibly singular, with a specified range of nonzero
eigenvalues. Nonzero values of $\eta$ and $\zeta$ induce different
types of orientation bias. In particular, we see that $\Lambda$
\eqref{eq:def:Lambda} increases toward its upper bound of $\sqrt{n}$
as $\zeta$ increases away from zero.

We tested three CD variants.
\begin{itemize}
\item CYCLIC: Cyclic CD, described in  Subsection~\ref{sec:cyclic}.
\item IID: Randomized CD using sampling with replacement:
  Algorithm~\ref{alg:rcd}.
\item EPOCHS: The ``sampling without replacement'' variant of
  Algorithm~\ref{alg:rcd}, described following the proof of
  Theorem~\ref{th:rcd}.
%% a version of randomized CD in which the coordinates are
%%   chosen in ``epochs'' each consisting of $n$ successive
%%   iterations. At the start of each epoch, the indices
%%   $\{1,2,\dotsc,n\}$ are shuffled, and at the $n$ iterations within
%%   the epoch, the indices $i_k$ are chosen according to the shuffled
%%   order. This scheme can be viewed as ``sampling without
%%   replacement,'' in contrast to the ``sampling with replacement''
%%   strategy of Algorithm~\ref{alg:rcd}. The approach has been popular
%%   in other contexts (particularly stochastic gradient algorithms) for
%%   some time.
\end{itemize}

For each variant, we tried both a fixed steplength $\alpha_k \equiv
1/\Lmax$ and the optimal steplength $\alpha_k = 1/Q_{i_k,i_k}$. Thus,
there were a total of six algorithmic variants tested.

The starting point $x^0$ was chosen randomly, with all components from
the unit normal distribution $N(0,1)$. The algorithms were terminated
when the objective was reduced by a factor of $10^{-6}$ over its
initial value $f(x^0)$. 
%% We ran ten trials for each setting of the
%% problem parameters, each involving new random constructions of $Q$ and
%% a different random choice of $x^0$.
%%  The statistics reported below are averages obtained over these ten
%% trials. (All are arithmetic averages except for $\cond(Q)$, which
%% is a geometric average.)

The speed of convergence varied widely according to the problem
construction parameters $\eta$, $\lambda$, and $\cond(\Sigma)$, but we
can make some general observations.  First, on problems that are not
well conditioned, the function values $f(x^k)$ decreased rapidly at
first, then settled into a {\em linear} rate of decrease. This linear
rate held even for problems in which $Q$ was singular --- a
significant improvement over the sublinear rates predicted by the
theory. Second, the EPOCHS variant of randomized CD tended to converge
faster than the IID version, though rarely more than twice as
fast. Third, the use of the optimal step was usually better than the
fixed step (with sometimes up to six times fewer iterations), but this
was by no means always the case. Fourth, while there were extensive
regimes of parameter values in which all six variants performed
similarly, there were numerous ``stressed'' settings in which the
CYCLIC variants are much slower than the randomized variants, by
factors of $10$ or more.

\section{Parallel CD Algorithms} 
\label{sec:parallel}

CD methods lend themselves to different kinds of parallel
implementation. Even basic algorithm frameworks such as
Algorithm~\ref{alg:cd} may be amenable to application-specific
parallelism, when the computations involved in evaluating a single
element of the gradient vector are substantial enough to be spread out
across cores of a multicore computer. We concern ourselves here with
more generic forms of parallelism, which involve multiple instances of
the basic CD algorithm, running in parallel on multiple processors.

We can distinguish different types of parallel CD algorithms. {\em
  Synchronous} algorithms are those that partition the computation
into pieces that can be executed in parallel on multiple processors
(or cores of a multicore machine), but that synchronize frequently
across all processors, to ensure consistency of the information
available to all processors at certain points in time. For example,
each processor could update a subset of components of $x$ in parallel
(with the subsets being disjoint), and the synchronization step could
ensure that the results of all updates are shared across all
processors before further computation occurs.  The synchronization
step often detracts from the performance of algorithms, not only
because some processors may be forced to idle while others complete
their work, but also because the overheads associated with (hardware
and software) locking of memory accesses can be high. Thus, {\em
  asynchronous} methods, which weaken or eliminate the requirement of
consistent information across processors, are preferred in practice.
Analysis of such methods is more difficult, but results have been
obtained that accord with practical experience of such
methods. Indeed, it can be verified that in certain regimes, linear
speedup can be expected across a modest number of processors.

\subsection{Synchronous Parallelism} \label{sec:parallel.sync}

We mention several synchronous parallel variants of CD that appear in
the recent literature. We note that in the some of these papers, the
computational results were obtained by implementing the methods in an
asynchronous fashion, disregarding the synchronization step required
by the analysis.

Bradley at al.~\cite{Bra11a} consider a bound-constrained problem that
is a reformulation of the problem \eqref{eq:reg} with specific choices
of $f$ and with $\Omega(x) = \| x\|_1$. Their algorithm performs
short-step updates of individual components of $x$ in parallel on $P$
processors, with synchronization after each round of parallel
updating. This scheme essentially updates a randomly-chosen block of
$P$ variables at each cycle. By modifying the analysis of
\cite{ShaT11a}, they show that the $1/k$ sublinear convergence rate
bound is not affected provided that $P$ is no larger than $n/L$, where
$L$ is the Lipschitz constant from \eqref{eq:L}.

Jaggi et al.~\cite{Jag14d} perform a synchronized CD method on the
dual ERM model \eqref{eq:erm.dual} for the case of $g(w) = g^*(w) =
(1/2) \|w\|^2$, partitioning components of the dual variable $x$
between cores and sharing a copy of the vector $Ax$ across cores,
updating this vector at each synchronization point. The approach can
be thought of as a nonlinear block Gauss-Jacobi method (by contrast
with the coordinate Gauss-Seidel approaches discussed in
Section~\ref{sec:algs}).

Richtarik and Takac~\cite{RicT12c} describe a method for the separably
regularized formulation \eqref{eq:reg}, \eqref{eq:reg.sep} in which a
subset of indices $S_k \subset \{1,2,\dotsc,n\}$ is updated according
to the formula in Algorithm~\ref{alg:cdreg}. The work of updating the
components in $S_k$ is divided between processors; essentially, a
synchronization step takes place at each iteration. This scheme is
enhanced with an acceleration step in \cite{FerQRT14}; the extra
computations associated with the acceleration step too are
parallelized, using ideas from \cite{LeeS13}.  In the scheme of
Marecek, Richtarik, and Takac~\cite{MarRT14a}, the variable vector $x$
is partitioned into subvectors, and each processor is assigned the
responsibility for updating one of these subvectors.  On each
processor, the updating scheme described in \cite{RicT12c} is applied,
providing a second level of parallelism. Synchronization takes place
at each outer iteration. Details of the information-sharing between
processors required for accurate computation of gradients in different
applications are described in \cite[Section~6]{MarRT14a}.

\subsection{Asynchronous Parallelism} \label{sec:parallel.async}

In asynchronous variants of CD, the variable vector $x$ is assumed to
be accessible to each processor, available for reading and
updating. (For example, $x$ could be stored in the shared-memory space
of a multicore computer, where each core is viewed as a processor.)
Each processor runs its own CD process, shown here as
Algorithm~\ref{alg:pcd.1}, without any attempt to coordinate or
synchronize with other processors.  Each iteration on each processor
chooses an index $i$, loads the components of $x$ that are needed to
compute the gradient component $[\nabla f(x)]_{i}$, then updates the
$i$th component $x_i$. Note that this evaluation may need only a small
subset of the components of $x$; this is the case when the Hessian
$\nabla^2 f$ is structurally sparse, for example. On some multicore
architectures (for example, the Intel Xeon), the update of $x_i$ can
be performed as a unitary operation; no software or hardware locking
is required to block access of other cores to the location $x_i$.
%% , so there is no danger that the value of $x_i$ will be changed by
%% some other core during the update step.

\begin{algorithm} 
\caption{Coordinate Descent for \eqref{eq:f} (running on each Processor)\label{alg:pcd.1}}
\begin{algorithmic}
\Repeat
\State Choose index $i \in \{1,2,\dotsc,n\}$;
\State Evaluate $[\nabla f(x)]_{i}$, reading  components of $x$ from shared memory as necessary;
\State Update $x_i \leftarrow x_i - \alpha [\nabla f(x)]_{i}$ for some $\alpha>0$;
\Until termination;
\end{algorithmic}
\end{algorithm}

We can take a global view of the entire parallel process, consisting
of multiple processors each executing Algorithm~\ref{alg:pcd.1}, by
defining a global counter $k$ that is incremented whenever {\em any}
processor updates an element of $x$: see Algorithm~\ref{alg:pcd.2}. Note
that the {\em only} difference with the basic framework of
Algorithm~\ref{alg:cd} is in the argument of the gradient component:
In Algorithm~\ref{alg:cd} this is the latest iterate $x^k$ whereas in
Algorithm~\ref{alg:pcd.2} it is a vector $\hat{x}^k$ that is generally
made up of components of vectors from previous iterations $x^j$,
$j=0,1,\dotsc,k$. The reason for this discrepancy is that between the
time at which a processor {\em reads} the vector $x$ from shared
storage in order to calculate $[\nabla f(x)]_i$, and the time at which
it {\em updates} component $i$, {\em other processors} have generally
made changes to $x$. In consequence, each update step is using
slightly stale information about $x$. To prove convergence results, we
need to make assumptions on how much ``staleness'' can be tolerated,
and to modify the convergence analysis quite substantially. Indeed,
proofs of convergence even for the most basic asynchronous algorithms
are quite technical.

\begin{algorithm} 
\caption{Asynchronous Coordinate Descent for \eqref{eq:f}\label{alg:pcd.2}}
\begin{algorithmic}
\State Set $k \leftarrow 0$ and choose $x^0 \in \R^n$;
\Repeat
\State Choose index $i_k \in \{1,2,\dotsc,n\}$;
% and set $d^k = [\nabla f(x^k)]_{i_k} e_{i_k}$; 
\State $x^{k+1} \leftarrow x^k - \alpha_k [\nabla f(\hat{x}^k)]_{i_k} e_{i_k}$ for some $\alpha_k>0$;
\State $k \leftarrow k+1$;
\Until termination test satisfied;
\end{algorithmic}
\end{algorithm}

Asynchronous CD algorithms are distinguished from each other mostly by
the assumptions they make on the the choice of update components $i_k$
and on the ``ages'' of the components of $\hat{x}^k$, that is, the
iterations at which each component of this vector was last updated. In
the terminology of Bertsekas and Tsitsiklis~\cite{BerT89}, the
algorithm is {\em totally asynchronous} if
\begin{itemize}
\item[(a)] each index $i \in
\{1,2,\dotsc,n\}$ of $x$ is updated at infinitely many iterations; and
\item[(b)] if $\nu^k_j$ denotes the iteration at which component $j$ of
the vector $\hat{x}^k$ was last updated, then $\nu^k_j \to \infty$ as $k
\to \infty$ for all $j=1,2,\dotsc,n$.
\end{itemize}
In other words, each component of $x$ is updated infinitely often, and
all components used in successive evaluation vectors $\hat{x}^k$ are
also updated infinitely often. 

The following convergence result for totally asynchronous variants of
Algorithm~\ref{alg:pcd.2} is due to Bertsekas and Tsitsiklis; see in
particular \cite[Sections~6.1, 6.2, and 6.3.3]{BerT89}.
\begin{theorem} \label{th:async}
Suppose that the problem \eqref{eq:f} has a unique solution $x^*$ and
that $f$ is convex and continuously differentiable. Suppose that
Algorithm~\ref{alg:pcd.2} is implemented in a totally asynchronous
fashion. Suppose that the mapping $T$ defined by $T(x) := x-\alpha
\nabla f(x)$ for some $\alpha>0$ (for which $x^*$ is the unique fixed
point) is strictly contractive in the $\ell_{\infty}$ norm, that is,
\begin{equation} \label{eq:contraction.max}
\| T(x) - x^* \|_{\infty} \le \eta \| x-x^*\|_{\infty}, \quad
\mbox{for some $\eta \in (0,1)$.}
\end{equation}
Then if we set $\alpha_k \equiv \alpha$ in Algorithm~\ref{alg:pcd.2},
the sequence $\{ x^k \}$ converges to $x^*$.
\end{theorem}
We cannot expect to obtain a convergence rate in this setting (such as
sublinear with rate $1/k$), given that the assumptions on the ages of
the components in $\hat{x}^k$ are so weak. Although this result can be
generalized impressively and its proof is not too complex, we should
note that the $\ell_{\infty}$ contraction assumption
\eqref{eq:contraction.max} is quite strong. It is violated even by
some strictly convex objectives $f$. For example, when $f(x) =
(1/2)x^TQx$ with
\[
Q=\left[ \begin{matrix} 1 & 1 \\ 1 & 2 \end{matrix} \right],
\]
we have $f$ strictly convex with minimizer $x^*=0$. However the
mapping $T(x) = (I-\alpha Q)x$ is not contractive for any $\alpha>0$;
we have for example that $\|T(x) \|_{\infty} \ge \|x\|_{\infty}$ when
$x=(1, -1)^T$.

We turn now to {\em partly asynchronous} variants of
Algorithm~\ref{alg:pcd.2}, in which we make stronger assumptions on
the ages of the components of $\hat{x}^k$. Liu and
Wright~\cite{LiuW14c} consider a version of Algorithm~\ref{alg:pcd.2}
that is the parallel analog of Algorithm~\ref{alg:rcd}, in that each
update component $i_k$ is chosen independently and randomly with equal
probability from $\{1,2,\dotsc,n \}$. They assume that no component of
$\hat{x}^k$ is older than a nonnegative integer $\tau$ --- the
``maximum delay'' --- for any $k$. Specifically, they express the
difference between $x^k$ and $\hat{x}^k$ in terms of ``missed
updates'' to $x$, as follows:
\begin{equation} \label{eq:ji.1}
x^k = \hat{x}^k + \sum_{l \in K(j)} (x^{l+1}-x^l),
\end{equation}
where $K(j)$ is a set of iteration numbers drawn from the set $\{ j-q
\, : \, q=1,2,\dotsc, \tau \}$.  The value of $\tau$ is related to the
number of processors $P$ involved in the computation. If all
processors are performing their updates at approximately the same
rates, we could expect $\tau$ to be a modest multiple of $P$ ---
perhaps $\tau = 2P$ or $\tau=3P$, to allow a safety margin for
occasional delays. Hence the value of $\tau$ is an indicator of
potential parallelism in the algorithm.

In \cite{LiuW14c}, the steplengths in Algorithm~\ref{alg:pcd.2} are
fixed as follows:
\begin{equation} \label{eq:ji.4}
\alpha_k \equiv \frac{\gamma}{\Lmax},
\end{equation}
where $\gamma$ is chosen to ensure that Algorithm~\ref{alg:pcd.2}
progresses steadily toward a solution, but not too rapidly. Too-rapid
convergence would cause the information in $\hat{x}^k$ to become too
stale too quickly, so the gradient component $[\nabla
  f(\hat{x}^k)]_{i_k}$ would lose its relevance as a suitable update
for the variable component $x_{i_k}$ at iteration $k$. Steady
convergence is enforced by choosing some $\rho>1$ and requiring that
\begin{equation} \label{eq:ji.2}
E \| x^{k-1} - \bar{x}^k \|^2 \le \rho  E \| x^k - \bar{x}^{k+1} \|^2,
\end{equation}
where $\bar{x}^k$ is the vector that would hypothetically be obtained
if we were to apply the the update to {\em all} components, that is,
\[
\bar{x}^{k+1} := x^k - \frac{\gamma}{\Lmax} \nabla f(\hat{x}^k),
\]
and the expectations $E(\cdot)$ are taken over all random variables
$i_0,i_2,\dotsc$. Condition \eqref{eq:ji.2} ensures that the
``expected squared update norms'' decrease by at most a factor of
$1/\rho$ at each iteration.

The main results in \cite{LiuW14c} apply to composite functions
\eqref{eq:reg}, \eqref{eq:reg.sep}, but for simplicity here we state
the result in terms of the problem \eqref{eq:f}, where $f$ is convex
and continuously differentiable, with nonempty solution set $\cS$ and
optimal objective value $f^*$. We use $P_{\cS}$ to denote projection
onto $\cS$, and recall the definition \eqref{eq:def:Lambda} of the
ratio $\Lambda$ between different varieties of Lipschitz
constants. The results also make use of an {\em optimal strong
  convexity} condition, which is that the following inequality holds
for some $\sigma>0$:
\begin{equation} \label{eq:osc}
f(x) - f^* \ge \frac{\sigma}{2} \| x-P_{\cS}(x) \|^2, \quad
\mbox{for all $x$.}
\end{equation}

The following result is a modification of \cite[Corollary~2]{LiuW14c}.
\begin{theorem} \label{th:ji}
Suppose that Assumption~\ref{ass:fconv} holds, and that
\begin{equation} \label{eq:ji.3}
4 e \Lambda(\tau+1)^2 \le \sqrt{n}.
\end{equation}
Then by setting $\gamma=1/2$ in \eqref{eq:ji.4} (that is, choosing
steplengths $\alpha_k \equiv 1/(2 \Lmax)$), we have that
\begin{equation} \label{eq:ji.6}
E \left( f(x^k) \right) - f^*  \le \frac{n (\Lmax  \| x^0-P_{\cS}(x^0) \|^2 + f(x^0)-f^*)}{n+k}.
\end{equation}
Assuming in addition that \eqref{eq:osc} is satisfied for some
$\sigma>0$, we obtain the following linear rate:
\begin{align} 
\nonumber
E & \left( f(x^k) \right) - f^* \\
\label{eq:ji.5}
& \le
\left( 1- \frac{\sigma}{n(\sigma + 2\Lmax)} \right)^k
( \Lmax \|x^0-P_{\cS}(x^0) \|^2 + f(x^0) - f^* ).
\end{align}
\end{theorem}
A comparison with Theorem~\ref{th:rcd}, which shows convergence rates
for serial randomized CD (Algorithm~\ref{alg:rcd}) shows a striking
similarity in convergence bounds.  The factor-of-$2$ difference in
steplength between the serial and parallel variants accounts for most
of the difference between the linear rates \eqref{eq:rcd.linear} and
\eqref{eq:ji.5}, while there is an extra term $n$ in the denominator
of the sublinear rate \eqref{eq:ji.6}. We conclude that we do not pay
q high overhead (in terms of total workload) for parallel
implementation, and hence that near-linear speedup can be
expected. (Indeed, computational results in \cite{LiuW14c} and
\cite{LiuW13a} observe near-linear speedup for multicore asynchronous
implementations.) 

These encouraging conclusions depend critically on the condition
\eqref{eq:ji.3}, which is an upper bound on the allowable delay $\tau$
in terms of $n$ and the ratio $\Lambda$ from \eqref{eq:def:Lambda}.
For functions $f$ with weak coupling between the components of $x$
(for example, when off-diagonals in the Hessian $\nabla^2 f(x)$ are
small relative to the diagonals), we have $\Lambda$ not much greater
than $1$, so the maximum delay can be of the order of $n^{1/4}$ before
there is any attenuation of linear speedup. When stronger coupling
exists, the restriction on $\tau$ may be quite tight, possibly not
much greater than $1$. A more general convergence result
\cite[Theorem~1]{LiuW14c} shows that in this case, we can choose
smaller values of $\gamma$ in \eqref{eq:ji.4}, allowing graceful
degradation of the convergence bounds while still obtaining fairly
efficient parallel implementations.

We note that an earlier analysis in \cite{LiuW13a} made a stronger
assumption on $\hat{x}^k$ --- that it is equal to some earlier iterate
$x^j$ of Algorithm~\ref{alg:pcd.2}, where $k \ge j \ge k-\tau$, that
is, the earlier iterate is no more than $\tau$ cycles old. (A similar
assumption was used to analyze convergence of as asynchronous
stochastic gradient algorithm in \cite{Hogwild-nips}.) This stronger
assumption yields stronger convergence results, in that the bound on $\tau$ in
\eqref{eq:ji.3} can be loosened.  However, the assumption may not
always hold, since some parts of $x$ in memory may be altered by some
cores as they are being read by another core, a phenomenon referred to
in \cite{LiuW14c} as ``inconsistent reading.''

\section{Conclusion}
\label{sec:conclusion}

We have surveyed the state of the art in convergence of coordinate
descent methods, with a focus on the most elementary settings and the
most fundamental algorithms. The recent literature contains many
extensions, enhancements, and elaborations;
% . We mention in particular
% \cite[Section~6]{Nes14d}
\nocite{Nes14d} 
we refer interested readers to the bibliography of this
paper, and note that new works are appearing at a rapid pace.

Coordinate descent method have become an important tool in the
optimization toolbox that is used to solve problems that arise in
machine learning and data analysis, particularly in ``big data''
settings. We expect to see further developments and extensions,
further customization of the approach to specific problem structures,
further adaptation to various computer platforms, and novel
combinations with other optimization tools to produce effective
``solutions'' for key application areas.

\begin{acknowledgements}
I thank Ji Liu for the pleasure of collaborating with him on this
topic over the past two years. I am grateful to the editors and
referees of the paper, whose expert and constructive comments led to
numerous improvements.
\end{acknowledgements}

%\bibliographystyle{spmpsci} % mathematics and physical sciences
%\bibliography{reference,refs}

% \input{wright-ismp_v2_extras.tex}

\end{document}